\documentclass{amsart}

\usepackage[all]{xy}
\usepackage{graphicx, overpic, epsf,amssymb,amscd,amsfonts,times, enumerate, amsthm, amstext, amsmath, xypic, subfigure}

\graphicspath{{OldPics/}{ThesisPics/}}

\newtheorem{theorem}{Theorem}[section]

\newtheorem{proposition}[theorem]{Proposition}
\newtheorem{definition}[theorem]{Definition}
\newtheorem{lemma}[theorem]{Lemma}
\newtheorem{corollary}[theorem]{Corollary}

\newcommand{\wh}{\widehat} 

\newcommand{\bZ}{\mathbb Z}

\numberwithin{equation}{section}

\title[]{Characterizing Right-Veering Homeomorphisms\\
       of the Punctured Torus via\\
       the Burau Representation\\}

\author{Emille K. Davie}

\begin{document}

\begin{abstract}

We classify right-veering homeomorphisms of the once-punctured torus using the Burau representation of the 3-strand braid group.  We show that reducible and periodic mapping classes in $B_3$ can be identified as right-veering by consideration of the reduced version of the Burau representation.  Given any element $\beta$ in $B_3$, we give a method to quickly determine its action on the generators of the fundamental group of the $3$-times punctured disk.  This action which determines whether $\beta$ is right-veering, left-veering, or neither. 

\end{abstract}

\maketitle  

\pagenumbering{arabic}  

\section{Introduction and Background}

The goal of this paper is to use the Burau matrix of a $3$-braid to determine whether the corresponding homeomorphism $h$ is right-veering, left-veering, or neither.  In Section 1, we give brief introductions to all relevant ideas used throughout this paper.   In Section 2, we explain the veering information of a reducible or periodic element of $B_3$ by consideration of the trace of a power of its reduced Burau matrix.  Finally, in Section 3, we prove the following main theorems that allow us to determine a $3$-braid's action on the generators of the $3$-punctured disk by consideration of their geometric intersection with arcs $\alpha_1,\alpha_2,\alpha_3$ that start and end at the punctures. 

\begin{theorem}\label{main} Given a 3-braid $\beta$ and its matrix $\mathrm M(\beta)$ in the image of the unreduced Burau representation, we can determine the geometric intersection number $i(\beta(y_i),\alpha_j)$ of $\beta(y_i)$ and  $\alpha_j$ from the entries of $M(\beta)$ for each $1\leq i,j\leq 3$ .\end{theorem}
\begin{theorem} \label{same}If $\beta_1$ and $\beta_2$ are 3-braids and $i(\beta_1(y_i),\alpha_j)=i(\beta_2(y_i),\alpha_j)$  for each $i$ and $j$, then $\beta_1=\Delta^{2k}\beta_2$ for some integer $k$.\end{theorem}

\noindent We start with a few definitions.

\subsection{The Braid Group}
The braid group $B_n$ can be defined as the group generated by $\sigma_1, \sigma_2,\dots,\sigma_{n-1}$ with relations:
\begin{enumerate}
\item $\sigma_i\sigma_j=\sigma_j\sigma_i$ if $|i-j|> 1$, and
\item $\sigma_i\sigma_j\sigma_i=\sigma_j\sigma_i\sigma_j$ if $|i-j|=1$.
\end{enumerate}
The braid group can also be described as a subgroup of the automorphism group of a free group or as a group of geometric objects up to isotopy.  In this paper we will use the interpretation of $B_n$ as the mapping class group of the space $D_n$, where $D_n$ is the unit disk in $\mathbb R^2$ with the set $Q=\{p_1,p_2,\dots,p_n\}$ of $n$ interior points of $D^2$ (assumed to lie regularly spaced on the x-axis) excised.  We call these points \emph{punctures}.  We define the $n$-braid $\Delta_n$ to be $$(\sigma_1\sigma_2\dots\sigma_{n-1})(\sigma_1\sigma_2\dots\sigma_{n-2})\dots(\sigma_1\sigma_2)(\sigma_1).$$  The square of $\Delta_n$ generates the infinite cyclic center of $B_n$.  As a mapping class, $\Delta_n^2$ is a full positive Dehn twist about a curve parallel to $\partial D_n$.  We denote $\Delta_3$ as $\Delta$ throughout this paper.

Let $S$ be a genus one surface with one boundary component.   We will make use of the fact that the group $B_3$ is isomorphic to the relative mapping class group of $S$, denoted here as $\mathrm{Mod(S,\partial S)}$.  This isomorphism is given by sending generators $\sigma_1, \sigma_2$ to positive Dehn twists $T_x, T_y$ about dual closed curves $x$ and $y$ in $S$.  We will also rely on the homomorphism $\psi$ in the short exact sequence

 \begin{equation*}
\xymatrix{
1\ar[r] &\bZ\ar[r] & B_3\ar[r]^{\psi} & SL_2(\bZ)\ar[r] & 1.
}
\end{equation*}  

By consideration of the representation onto $SL_2(\bZ)$, we use the Thurston classification of surface diffeomorphisms (which can be found in \cite{8}) to determine whether a braid representative $h$ is periodic, reducible, or pseudo-Anosov.  We do this simply by calculating the trace.  Say $\hat h$ is the image of $h$ under $\psi$. If $|\mathrm{tr}(\hat h)|=2$, then $h$ is reducible.  If $|\mathrm{tr}(\hat h)|<2$, $h$ is periodic.  Finally, if $|\mathrm{tr}(\hat h)|>2$, $h$ is pseudo-Anosov.  We move forward using this classification to identify right-veering braids, however, first we must address the question of what a right-veering braid is.

\subsection{Right-Veering Surface Homeomorphisms}
Throughout,  let $\Sigma$ denote a compact, oriented surface with nonempty boundary.   Let $\alpha$ and $\beta$ be smooth properly embedded oriented arcs in $\Sigma$ with a common basepoint  $x\in \partial\Sigma$.  We also assume that $\alpha$ and $\beta$ intersect minimally and transversely.  Consider the vectors $v_{\alpha}$ and $v_{\beta}$ that are tangent to $\alpha$ and $\beta$, respectively, at $x$.  We say that $\beta$ is \itshape to the right \upshape of $\alpha$ if either $\alpha$ is isotopic to $\beta$ or $\{v_{\beta}, v_{\alpha}\}$ corresponds to the orientation of $\Sigma$ at $x$.  Alternatively, if $\pi: \tilde\Sigma\rightarrow\Sigma$ is the universal cover of $\Sigma$.  Let $\tilde x\in\partial \tilde\Sigma$ be a lift of $x$.  Consider lifts $\tilde \alpha$ and $\tilde\beta$ of $\alpha$ and $\beta$ with $\tilde\alpha(0)=\tilde\beta(0)=\tilde x$.  The arc $\tilde\alpha$ divides $\tilde\Sigma$ into two regions.  The region to the left where the boundary orientation induced from the region coincides with the orientation on $\tilde\alpha$ and the region to the right.  Here, we say that $\beta$ is \emph{to the right} of $\alpha$.  We now define the monoid of right-veering maps.

\begin{definition}Let $h$ represent an element of $\mathrm{Mod}(\Sigma,\partial\Sigma)$.  Then $h$ is right-veering if for every $x\in\partial\Sigma$ and every  properly embedded oriented arc $\alpha$ based at $x$, we have that $h(\alpha)$ is to the right of $\alpha$.  We denote the set of all such maps $\mathrm{Veer}(\Sigma,\partial\Sigma)=\mathrm{Veer}$.\end{definition}

Similarly, we say that $\beta$ is \itshape to the left \upshape of $\alpha$ if $\{v_{\alpha},v_{\beta}\}$ corresponds to the orientation of $\Sigma$ at $x$, and so we may also refer to a map $h$ as \textit{left-veering}. These homeomorphisms are especially helpful in contact topology.  In \cite{1} Honda, Kazez, Mati\'c prove that a contact 3-manifold $(M,\xi)$ is tight if and only if all of its compatible open book decompositions $(\Sigma,h)$ have $h\in \mathrm{Veer}$. Furthermore, they prove in \cite{2} that if $\Sigma$ has genus equal to one and one boundary component, $(M,\xi)$ is tight if and only if some compatible open book has right-veering monodromy.

\subsection{The Burau Representation}

Now that we have an handle on what a right-veering map is, we introduce the tool that we will use to identify right-veering $3$-braids: the Burau representation of $B_n$.  The Burau representation can be defined in several different ways.   The description given here comes from the viewpoint of covering spaces.   Again, we let $D_n$ be the $n$-punctured disc with basepoint $p_0$ on the boundary.  Identify $\pi_1(D_n,p_0)$ in the usual way with the free group $F_n$, and let $\epsilon: \pi_1(D_n,p_0)\rightarrow\bZ$ be the exponent sum map which takes the element $y_{r_1}^{\alpha_1}y_{r_2}^{\alpha_2}\dots y_{r_s}^{\alpha_s}$ to the integer $\sum_{i=1}^s\alpha_i$.  Let ($\tilde D_n, q$) be the (regular) covering space of $D_n$ which corresponds to the kernel of this homomorphism.  

Concretely, $\tilde D_n$ can be constructed by making a series of cuts from the boundary of $D_n$ to each puncture (cut along the arcs $\delta_i$ described in the next section), then making a stack of $\bZ$ copies of this cut open space.  At each component, glue the right side of each cut to the left side of the corresponding cut in the component directly above it.  We end up with a space that resembles a parking garage with $n$ ``ramps'' and infinitely many ``levels".

Since $\epsilon: \pi_1(D_n,p_0)\rightarrow\bZ$ is a surjection, the group of deck transformations is isomorphic to $\bZ=\langle\tau\rangle$.  Letting $\Lambda=\bZ[t,t^{-1}]$, the induced action $\tau_{\ast}$ of $\tau$ on $H_1(\tilde D_n)$ endows $H_1(\tilde D_n)$ with a $\Lambda$-module structure by setting $t\cdot\gamma=\tau_{\ast}(\gamma)$, for $\gamma \in H_1(\tilde D_n)$.  Viewed as a $\Lambda$-module, $H_1(\tilde D_n)$ is free of rank $n-1$.   A chosen representative $h: D_n\rightarrow D_n$ will lift uniquely to a homeomorphism $\tilde h:\tilde D_n\rightarrow \tilde D_n$ that pointwise fixes the fiber over $p_0$.  Subsequently, $\tilde h$ induces the $\Lambda$-module automorphism $\tilde h_{\ast}:H_1(\tilde D_n)\rightarrow H_1(\tilde D_n)$.  We now make the following definition given a basis $\mathcal{B}_r=\{\gamma_1,\gamma_2,\dots,\gamma_{n-1}\}$ of $H_1(\tilde D_n)$.
\begin{definition}Given an element $\beta$ of the braid group and an element  $h$ which represents it, the \emph{reduced Burau representation} is the homomorphism:
$$\rho_r: B_n\rightarrow GL_{n-1}(\Lambda)$$
\noindent where $\rho_r(\beta)$ equals the automorphism $\tilde h_{\ast}$.  Let $\mathrm M_r(\beta)$ denote the matrix for $\rho_r(\beta)$ with respect to the basis $\mathcal{B}_r$.  We call this the \emph{reduced Burau matrix} of $\beta$.
\end{definition}

In a similar way, we may define the unreduced Burau representation.  Let  $H_1(\tilde D_n,\tilde p_0)$ be the first homology group of $\tilde D_n$ relative to the subset $\tilde p_0=q^{-1}(p_0)$.  Here, $\tilde p_0$ denotes the full preimage of $p_0$.    We lift the set of generators $y_1, y_2,\dots, y_n$ of $\pi_1(D_n,p_0)$ to a set of generators for $H_1(\tilde D_n,\tilde p_0)$.  
Now $H_1(\tilde D_n,\tilde p_0)$ can be thought of as a free $\Lambda$-module of rank $n$.  With basis $\mathcal{B}=\{\tilde y_1,\tilde y_2,\dots,\tilde y_n\}$ of $H_1(\tilde D_n,\tilde p_0)$, we make the following definition.  

\begin{definition} Given an element of the braid group $\beta$ and a homeomorphism $h$ that represents it, the \textit{unreduced Burau representation} is the homomorphism

$$\rho:B_n\rightarrow GL_n(\Lambda)$$

\noindent where $\rho(\beta)$ equals the automorphism $\tilde h_{\ast}: H_1(\tilde D_n,\tilde p_0)\rightarrow H_1(\tilde D_n,\tilde p_0)$. Let $\mathrm M(\beta)$ denote the matrix for $\rho(\beta)$ with respect to the basis $\mathcal{B}$.  We call this the \emph{unreduced Burau matrix} of $\beta$.\end{definition}

It is a well-known fact that Burau is faithful for $n=3$ (see \cite{5}).  In \cite{4}, Moody showed that Burau is not faithful for $n\geq9$, and this was sharpened by Long and Paton in \cite{3} to $n\geq6$.  The most recent improvement was made by Bigelow in \cite{7} when it was shown that Burau is not faithful for $n=5$. 

\subsection{A Geometric Interpretation}\label{geometric}
Unreduced Burau matrices can be obtained in several different ways, but perhaps the most beautiful approach was given by Moody in \cite{4}.  In this section we will describe his method for computing unreduced Burau matrices.   
Let  $y_i$ and $\delta_j$ for $i,j=1,\dots,n$ be as shown in Figure~\ref{n-disk for unreduced Burau}.  The line $L$ divides the disk into upper and lower halves.  We let $\partial D_n^+$ denote the upper half of $\partial D_n$.
\begin{figure}[htbp]
\begin{center}
\includegraphics[width=1.5in]{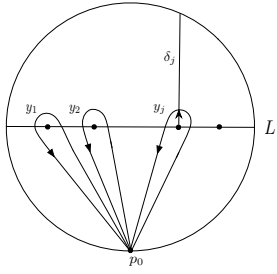}
\caption{The n-punctured disk}
\label{n-disk for unreduced Burau}
\end{center}
\end{figure}
Define a map $\langle\,\cdot,\cdot\rangle:H_1(\tilde D_n,\widetilde{\partial D_n^+}\cup\tilde Q)\times H_1(\tilde D_n,\tilde p_0)\rightarrow \Lambda$ as follows.  If $\tilde\alpha$ represents some homology class in $H_1(\tilde D_n,\tilde p_0)$, and $\tilde\delta_j$ is a lift of $\delta_j$, we define

$$\langle\delta_j,\alpha\rangle=\sum_{k\in\bZ}(t^k\tilde\delta_j,\tilde\alpha)t^k\in\Lambda$$ where $(t^k\tilde\delta_j,\tilde\alpha)$ is the algebraic intersection number of the two arcs in $\tilde D_n$.

It is clear that this pairing is only well-defined up to choice of the lifts and choice of orientation.  We orient $\delta_j$ from the puncture to the boundary and orient $y_i$ counterclockwise.  Then orientations of the lifts are induced from the arcs in the base space.  We use the following convention for choosing our lifts.  Choose component $\tilde p_0^{\star}$ of $\tilde p_0$, and declare the sheet of $\tilde D_n$ containing $\tilde p_0^{\star}$ to be the $0$-th level of $\tilde D_n$.  Now, define $\tilde y_j$ to be the unique lift of $y_j$ that starts at $\tilde p_0^{\star}$ and ends at $t\tilde p_0^{\star}$ and $\tilde \delta_j$ to be the unique lift of $\delta_j$ that intersects $\tilde y_j$.  Now we can write the action of any braid $\beta$ on basis vector $\tilde y_j$  of $H_1(\tilde D_n,\tilde p_0)$ as:
$$\rho(\beta)\tilde y_j = \Big(\langle\delta_1,\beta( y_j)\rangle\Big)\tilde y_1+\Big(\langle\delta_2,\beta(y_j)\rangle\Big)\tilde y_2+\dots+\Big(\langle\delta_n,\beta(y_j)\rangle\Big)\tilde y_n.$$

Using this method allows us to work in $D_n$ in a very concrete way instead of $\tilde D_n$.  The key is in keeping track of the level changes.  For example, we see the action of the braid $\sigma_1\in B_3$ on $\pi_1$ generators below in Figure~\ref{sigma_1}.
\begin{figure}[htbp]
\begin{center}
\includegraphics[height=1.25in]{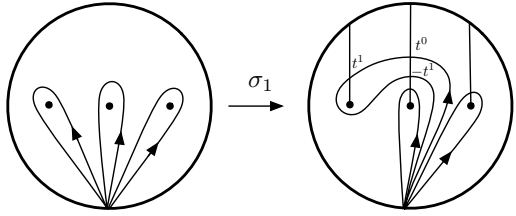}
\caption{The exponent of $t$ indicates the level on which the intersection occurs.}
\label{sigma_1}
\end{center}
\end{figure}
In particular, one can calculate the unreduced matrix for $\sigma_1\in B_3$ as $$\mathrm M(\sigma_1)=\begin{pmatrix}0&t&0\\ 1&1-t&0\\ 0&0&1\end{pmatrix}.$$  Similarly, we can calculate $\sigma_2\in B_3$ as $$\mathrm M(\sigma_2)=\begin{pmatrix}1&0&0\\ 0&0&t\\ 0&1&1-t\end{pmatrix}.$$  


\section{Reducible and Periodic Cases}

\indent In this section, we will use the reduced Burau representation to classify right-veering braids that are reducible or periodic.  We do so largely by appealing to work done by Honda, Kazez, and Mati\'c in~\cite{2} and by using the commutative diagram
 
 \begin{equation*}
\xymatrix{B_3\ar[r]^{\psi}\ar[d]_{\rho_r} & SL_2(\bZ)\\
\rho_r(B_3)\ar[ur]^{\pi}
}
\end{equation*}  
In this diagram, $\pi$ is specialization at $t=-1$, and $\psi$ the quotient map from $B_3$ to $B_3$ modulo its infinite cyclic center as before.  Geometrically, the homomorphism $\psi$ can be seen as collapsing the boundary of $S$ to a point.  For a braid $\beta$, let $\wh\beta$ denote its image under $\psi$ in $SL_2(\bZ)$.

\subsection{The Reducible Case}
We use a characterization of reducible maps of $S$ (see \cite{2}, for example) that says that any reducible map of $S$ can be written as $T_a^kT^m_b$, where $a$ is a boundary-parallel curve, and $b$ is some nonseparating curve in $S$.  Given the 2-fold branched cover of $D_3$ by $S$ (the branch set being the three points shown below), we see that a half twist about an arc from $p_i$ to $p_{i+1}$ in $D_3$ lifts to a full Dehn twist about some nonseparating curve in $S$.  
 
\begin{figure}[htbp]
\begin{center}
\includegraphics[height=2in]{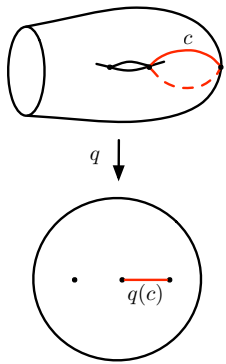}
\caption{$\sigma_2$ lifts to $T_c$}
\label{x_3}
\end{center}
\end{figure}

It follows that reducible elements of $B_3$ differ from a conjugate of a power of $\sigma_1$ by a powers of $\Delta^2$.  That is, any reducible braid $\omega$ can be written in the form $\Delta^{2k}\tau\sigma_1^m\tau^{-1}$, for some $\tau\in B_3$. We also use the following result of Honda, Kazez, and Mati\'c that yields the veering information of a reducible map of $S$ based upon the form above.

\begin{lemma}\label{hkm}(Honda, Kazez, Mati\'c \cite{2}) Let $h=T_a^kT^m_b$ represent a reducible mapping class in $\mathrm{Mod(S,\partial S)}$.  Then $h$ is right-veering if and only if either $k>0$ or $k=0$ and $m\geq0$. \end{lemma} 

Thus, given a reducible braid $\beta=\Delta^{2k}\tau\sigma_1^m\tau^{-1}$, our task is to extract the values $k$ and $m$ from the reduced Burau matrix $\mathrm M_r(\beta)$, then use Lemma~\ref{hkm} to determine whether $\beta$ is right-veering or not.  We make several statements about the trace of $\mathrm M_r(\beta)$.  In particular, we have the following proposition and corollaries which gives sufficient conditions for $\beta$ to be right-veering.  

\begin{proposition}\label{trace}Let $\beta \in B_3$ be reducible.  Then $\mathrm{tr}(\mathrm M_r(\beta))=t^{3k}+(-1)^mt^{3k+m}$.\end{proposition}
\begin{proof}
We have that the $\mathrm{tr}(\mathrm M_r(\sigma_1^m))=1+(-1)^mt^m=\mathrm{tr}(\mathrm M_r(\tau\sigma_1^m\tau^{-1}))$.  Now since $\mathrm M_r(\Delta^{2k})=\begin{pmatrix}t^{3k}&0\\0&t^{3k}\end{pmatrix}$, we may conclude that
$$
\mathrm{tr}(\mathrm M_r(\Delta^{2k}\tau\sigma_1^m\tau^{-1})) =t^{3k}(1+(-1)^mt^m)=t^{3k}+(-1)^mt^{3k+m}.
$$

\end{proof}

\begin{corollary}Let $\beta \in B_3$ be reducible. If $\mathrm{tr}(\mathrm M_r(\beta))=t^d-t^r$ for integers $d$ and $r$, then $d=3k$.  If $\mathrm{tr}(\mathrm M_r(\beta))=t^d+t^r$ is such that $3 \nmid r$, then $d=3k$.  In either of these cases, $\beta$ is right-veering if and only if $d>0$.\end{corollary}
\begin{proof} 
Immediate from Proposition~\ref{trace} and Lemma~\ref{hkm}.
\end{proof}
    
\begin{corollary}Let $\beta \in B_3$ be reducible.  
\begin{enumerate}
\item If $\mathrm{tr}(\mathrm M_r(\beta))=t^d+t^r$ for nonnegative integers $d$ and $r$, then $\beta$ is right-veering.
\item If $\mathrm{tr}(\mathrm M_r(\beta))=t^d+t^r$ for nonpositive integers $d$ and $r$, then $\beta$ is left-veering.
\end{enumerate}
\end{corollary}
\begin{proof}
If both $d$ and $r$ are nonnegative integers, then $k\geq 0$.  If $k>0$, then $\beta$ is right-veering.  If $k=0$, then $m$ is nonnegative which also yields $\beta$ right-veering.  An analogous argument is made if both $d$ and $r$ are nonpositive which gives $\beta$ left-veering.
\end{proof}

Now the question becomes how does one tell if a reducible braid is right-veering if the trace is of the form $t^d+t^{-r}$ where $d,r$ are positive integers both divisible by $3$.   For example, if $\mathrm{tr}(\mathrm M_r(\beta))=t^6+t^{-6}$, it is impossible to determine if $k=2$ and $m=-12$ or if $k=-2$ and $m=12$ by means already stated.  Moreover, this distinction is necessary since the former implies $\beta$ is right-veering and the latter that $\beta$ is left-veering.  The following theorem gives us sufficient information to determine $k$ and $m$ assuming $\mathrm{tr}(\mathrm M_r(\beta))$  and the eigenvalue $\lambda$ for $\wh\beta$ are known.
 
\begin{theorem}Suppose that $\beta$ is a nontrivial reducible map.  Then value $|m|$ is the greatest common divisor of the entries off the diagonal of $\wh\beta$.  Moreover, for $\lambda=1$, the sign of $m$ is the sign of the ($1,2)$-entry of $\wh\beta$, and for $\lambda=-1$, the sign of $m$ is the sign of the $(2,1)$-entry of $\wh\beta$.\end{theorem}
\begin{proof}
We first suppose that $\wh\beta$ has eigenvalue $\lambda=1$.  For any $\tau \in B_3$, it is easily checked that $\wh\beta$ is of the form $\begin{pmatrix} 1-amc & a^2m\\ -c^2m & 1+acm\end{pmatrix}$, where $\wh\tau=\begin{pmatrix}a&b\\ c&d\end{pmatrix}$.  Since $\wh\tau$ is in $SL_2(\bZ)$, we know that $gcd(a,c)=1$, and thus $gcd(a^2,c^2)=1$.  Therefore, the greatest common divisor of $a^2m$ and $-c^2m$ is $|m|$, and it is clear that the sign of $m$ determines the sign of $a^2m$.  
\\
Now if $\lambda=-1$, $\wh\beta$ is of the form $\begin{pmatrix}amc-1&-a^2m\\ c^2m&-1-acm\end{pmatrix}$.  The same argument holds except now the sign of $m$ determines the sign of $c^2m$.
\end{proof}

\subsection{The Periodic Case}

We use the fact that periodic maps of the torus have order dividing $12$, and start by stating a key lemma.  

\begin{lemma} \label{perlem}For any positive integer $n$, $h^n$ is right-veering if and only if $h$ is right-veering.\end{lemma}
\begin{proof}
It is clear that if $h$ is right-veering, then $h^n$ is right-veering since $\mathrm{Veer}$ is a monoid.  Now assume that $h^n(\alpha)$ is to the right of $\alpha$ for all properly defined arcs $\alpha$.  Then either $h(\alpha)$ is to the right of $\alpha$ for all relevant $\alpha$, or there is some arc $\gamma$ with $h(\gamma)$ to the left of $\gamma$.  However, the latter case implies that $h^n(\gamma)$ is to the left of $\gamma$.  This is a contradiction.
\end{proof}  

We have the following theorem which gives necessary and sufficient conditions for $\beta^{12}$ to be right-veering.

\begin{theorem}\label{perthm}Let $\beta\in B_3$ be a nontrivial periodic map.  Then $\mathrm M_r(\beta^{12})=\begin{pmatrix}t^{6k} & 0\\ 0 & t^{6k}\end{pmatrix}$ for some integer $k$.  Moreover, $\beta^{12}$ is right-veering if and only if $k\geq0$.\end{theorem} 
\begin{proof}
If $\beta$ is a periodic braid, then $\wh\beta$ is a periodic element of $SL_2(\bZ)$.  Since all periodic maps of $T^2$ have order dividing 12, $\beta^{12}$ must be an element of $ker(\psi)$ which is generated by $\Delta^4$.  We know that $\Delta^{2k}$ are powers of Dehn twists about the boundary of $D_3$, where this twist is positive if and only if $k$ is positive.  Since $\mathrm M_r(\Delta^{2k})=\begin{pmatrix}t^{3k}&0\\0&t^{3k}\end{pmatrix}$, we have that $\mathrm M_r(\beta^{12})=\begin{pmatrix}t^{6k}&0\\ 0&t^{6k}\end{pmatrix}$.  Moreover, this element is a positive Dehn twist if and only if $k>0$ and the identity element if and only if $k=0$.\end{proof}
 
\begin{corollary}Let $\beta$ be a periodic element of $B_3$.  Then $\beta$ is right-veering if and only if $\mathrm M_r(\beta^{12})$ has trace equal to $2t^s$ for some nonnegative integer $s$.\end{corollary}
\begin{proof}The proof is immediate from the Lemma~\ref{perlem} and Theorem~\ref{perthm}.\end{proof}


\section{The General Case}\label{general}

In this section we use the unreduced matrix $\mathrm M(\beta)$ of a $3$-braid $\beta$ to determine its veering information.  In short, given $\mathrm M(\beta)$, we use its entries to calculate geometric intersection numbers of $\beta(y_i)$ and arcs $\alpha_j$ for $1\leq i,j\leq 3$ shown below in Figure~\ref{alphas}.  This allows us to determine the action of $\beta$ on each $y_i$ up to powers of $\Delta^2$.  We can then write down a matrix $\mathrm M(\beta')$ that represents this action using the method described in Section~\ref{geometric}.  The matrix $\mathrm M(\Delta^{2k})\mathrm M(\beta')$ will be $\mathrm M(\beta)$, for exactly one integer $k$ because of the faithfulness of the representation for $n=3$.  As long as $\mathrm M(\beta')$ is well-chosen, the sign of $k$ will tell us the veering information.  We begin with the following theorem.     
\begin{figure}[htbp]
\begin{center}
\includegraphics[width=1in]{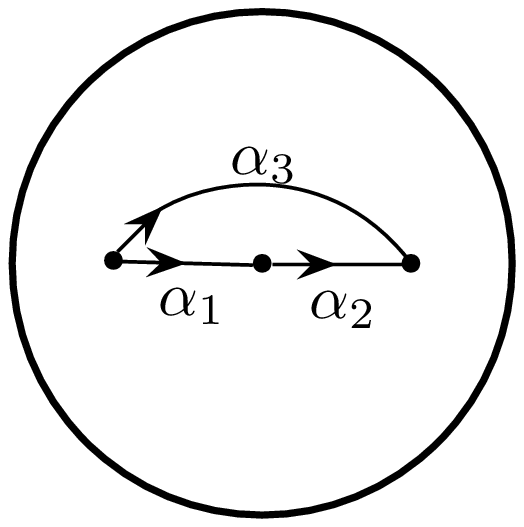}
\caption{}
\label{alphas}
\end{center}
\end{figure}

\begin{theorem}\label{pi_1} Let $\beta\in B_n$.  Then $\beta$ is right-veering if and only if $\beta$ sends every loop representing an element of $\pi_1(D_n,p_0)$ to the right.\end{theorem}

\begin{proof}Assume $\beta$ sends every loop representing an element of the fundamental group of $D_n$ to the right.  We endow $D_n$ with a hyperbolic metric for which $\partial D_n$ is geodesic.  The universal cover $\pi:\tilde X\rightarrow D_n$ is a subset of the Poincare disk $D^2$.  Consider the closure of $\tilde X$ in $D^2$, and denote one component of $\partial\tilde X$ as $L$.  Without any loss of generality, we let $\gamma$ be an oriented arc in $D_n$ with endpoints $p_0=\gamma(0)$ and $x_1=\gamma(1)$ both in $\partial D_n$.  We also assume that $\gamma$ is essential in the sense that it does not cobound a digon with some subarc of $\partial D_n$.  Then $\gamma\cup\alpha$ is a nontrivial loop based at $p_0$, for $\alpha$ an embedded arc in $\partial D_n$ oriented clockwise from $x_1$ to $p_0$.  Choose $\tilde p_0 \in \pi^{-1}(p_0)$ in $L$, and let $\tilde \beta:\tilde X\rightarrow\tilde X$ be a lift of $\beta$ which fixes $L$ pointwise.  

Choose a lift $\widetilde{\gamma\cup\alpha}=\tilde\gamma\cup\tilde\alpha$ of $\gamma\cup\alpha$ which has initial point $\tilde p_0$ and terminal point $\tilde y_0=g(\tilde p_0)$, for $g\in \pi_1(D_n,p_0)$.  Without loss of generality, we assume that $\tilde \gamma$ is geodesic.  Note that $\tilde\gamma(1)=\tilde x_1=\tilde\alpha(0)$.  
By hypothesis $\tilde \beta(\tilde \gamma\cup\tilde\alpha)$ is in the region to the right of $\tilde\gamma\cup\tilde\alpha$.  Because $\pi^{-1}(\alpha)$ is a collection of disjoint arcs in $\partial \tilde X$, we can conclude that $\tilde \beta(\tilde\alpha)\cap\tilde\alpha=\emptyset$.  Therefore, $\tilde\beta(\tilde\alpha)$ is in the region to the right of $\tilde\alpha$ and hence, $\tilde\beta(\tilde\gamma)$ (having terminal point in $\tilde\beta(\tilde\alpha)$) is in the region to the right of $\tilde\gamma$.  The converse implication is clearly true by definition of $\mathrm{Veer}$.
\end{proof}

Given Theorem ~\ref{pi_1}, our goal is to detect rightward movement of generators $y_i$ of the fundamental group of $D_n$.  
\subsection{Proof of Theorem~\ref{main}}

Let $\beta$ be a 3-braid and suppose that $C_i=(p_{1,i}(t),p_{2,i}(t),p_{3,i}(t))$ is the $i$-th column $\mathrm M(\beta)$.  Recall that $p_{j,i}(t)=\langle\delta_j,\beta(y_i)\rangle$.  Define the following Laurent polynomials: $$Q_i(t)=p_{1,i}(t)-tp_{2,i}(t),\\ R_i(t)=p_{2,i}(t)-tp_{3,i}(t), \\ S_i(t)=p_{1,i}-t^2p_{3,i}(t).$$  

\noindent We claim that  $Q_i, R_i,S_i$ are intersection pairings of $\beta(y_i)$ and $\alpha_1, \alpha_2,\alpha_3$, respectively.  
To see this, we first note that $$Q_i(t)=\langle \delta_1,\beta(y_i)\rangle-t\langle \delta_2,\beta(y_i)\rangle=\langle\delta_1-t^{-1}\delta_2,\beta(y_i)\rangle=\sum_{k\in\bZ}(t^k(\tilde\delta_1-t^{-1}\tilde\delta_2),\widetilde{\beta(y_i)})t^k.$$  If  $\tilde\alpha_1$ is a lift of $\alpha_1$ to the $0$-th level of $\tilde D_n$, then $\tilde \delta_1-t^{-1}\tilde\delta_2$ and $\tilde \alpha_1$ represent the same element of $H_1(\tilde D_n, \widetilde{\partial D_n^+}\cup\tilde Q)$.  Scaling by $t^{-1}$ is necessary because pushing $\tilde \delta_2$ off to the left results in a level change upwards.  Similarly, $\tilde \delta_2-t^{-1}\tilde\delta_3$ and $\tilde \alpha_2$ represent the same element of $H_1(\tilde D_n, \widetilde{\partial D_n^+}\cup\tilde Q)$.  Finally, if we say that $\tilde \alpha_3$ is the lift of $\alpha_3$ that begins on the $0$-th level of $\tilde D_n$ (and, hence, ends on the $-1$st level), then it is clear that $\tilde \alpha_3=\tilde \alpha_1+t^{-1}\tilde\alpha_2$ as homology classes.  We have $$\tilde\alpha_3=(\tilde\delta_1-t^{-1}\tilde\delta_2)+t^{-1}(\tilde\delta_2-t^{-1}\tilde\delta_3)=\tilde\delta_1-t^{-2}\tilde\delta_3 \in H_1(\tilde D_n, \widetilde{\partial D_n^+}\cup\tilde Q).$$ Thus, our formula for $S_i(t)$ gives us an intersection pairing of $\alpha_3$ and $\beta(y_i)$.

It remains to show that there is no cancellation of terms in this pairing, and hence we are able to calculate the geometric intersection numbers of the aforementioned arcs by specializing each polynomial at $t=-1$.  Let $w_r$ be the $r$-th intersection of $\beta(y_i)$ and $\alpha_1$, and consider the loop that is the union of the arcs from $w_r$ to $w_{r+1}$ along $\alpha_1$ and along $\beta(y_i)$.  If $\beta(y_i)$ intersects $\alpha_1$ in the same direction at both $w_r$ and $w_{r+1}$, then this loop must contain exactly two of the punctures (this is a direct consequence of there being three punctures).  Hence, if at $w_r$ the lift $\widetilde{\beta(y_i)}$ is on level $k$ of $\widetilde D_3$, then at $w_{j+1}$ $\widetilde{\beta(y_i)}$ is on level $k\pm2$ of $\widetilde D_3$.  The arc $\beta(y_i)$ will intersect $\delta_1$ or $\delta_2$ either just before or just after each intersection with $\alpha_1$ depending upon its orientation at $w_r$ and $w_{r+1}$.  Let's suppose $\beta(y_i)$ is oriented upward at $w_r$ and $w_{r+1}$.  Then there are three options for the diagram near $w_r$ and $w_{r+1}$ which are shown in Figure~\ref{choices1}.  The labeling in this figure reflects the monomial that is recorded in the row corresponding to either $\delta_1$ or $\delta_2$.  
\begin{figure}[htbp]
\begin{center}
\subfigure[]{\label{samecase1}\includegraphics[height=1in]{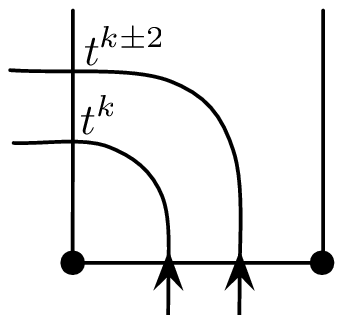}} 
\subfigure[]{\label{samecase3}\includegraphics[height=1in]{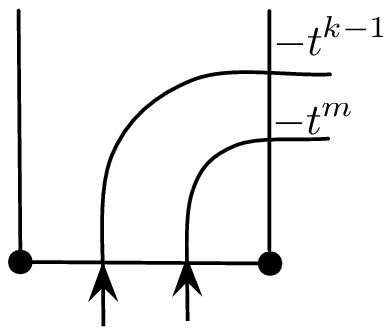}} 
\subfigure[]{\label{samecase2}\includegraphics[height=1in]{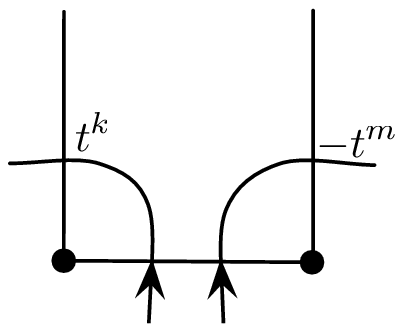}}
\end{center} \caption{$m$ is either $k-3$ or $k+1$.}
\label{choices1}
\end{figure} 
In all three instances, when we consider the monomial associated to $\delta_1$ and $-t$ times the monomial associated to $\delta_2$, there is a consistency between the parity of the exponent of $t$ and the sign of $t$ (meaning same parity, same sign and opposite parity, opposite sign).  As a result, specialization at $t=-1$ gives the accurate geometric intersection number of $2$ in absolute value.  A similar argument holds if $\beta(y_i)$ is oriented downward at $w_r$ and $w_{r+1}$.  

If $\beta(y_i)$ is oriented oppositely at $w_r$ and $w_{r+1}$, then the loop described above must either contain exactly one puncture or possibly all three.  Hence, if at $w_r$, the lift $\widetilde{\beta(y_i)}$ is on level $k$, then at $w_{r+1}$ it is either on level $k\pm 1$ or $k\pm 3$.  The local picture looks like one of the three options shown in Figure~\ref{choices2}.  Again, we see consistency between parity of the exponent and the sign of $t$ in the monomial associated to $\delta_1$ and $-t$ times the monomial associated to $\delta_2$.  So, indeed, $|Q_i(-1)|$ is the correct geometric intersection number.  The same argument holds for $|R_i(-1)|$.
\begin{figure}[htbp]
\begin{center}
\subfigure[]{\label{differentcase1}\includegraphics[height=1in]{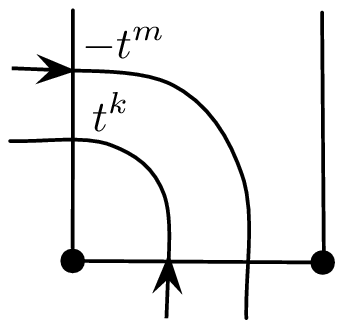}} 
\subfigure[]{\label{differentcase3}\includegraphics[height=1in]{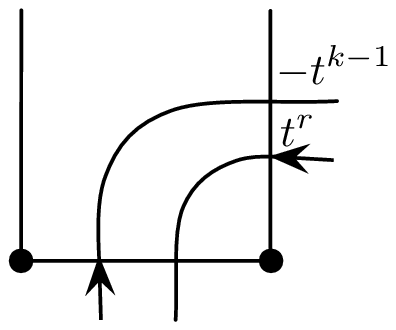}} 
\subfigure[]{\label{differentcase2}\includegraphics[height=1in]{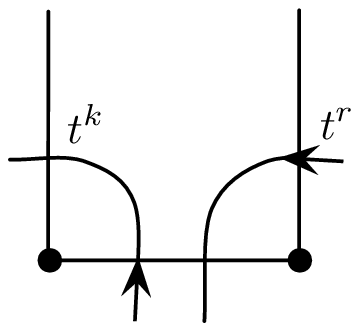}}
\end{center} \caption{$m$ is either $k\pm1$ or $k\pm3$ and $r$ is either $k,k-4$, or $k\pm2$.}
\label{choices2}
\end{figure} 

Lastly, the polynomials in the rows corresponding to $\delta_1$ and $\delta_3$ are inconsistent in the following sense:  if the terms in the polynomial corresponding to $\delta_1$ having even (odd) exponents are positive (negative), then the terms in the polynomial corresponding to $\delta_3$ having even (odd) exponents are negative (positive).  The only exceptions are the terms corresponding to intersections of the types seen below in Figure~\ref{exception}.  However, multiplying the polynomial corresponding to $\delta_3$ by $-t^2$ changes the any inconsistency to a consistency with the polynomial corresponding to $\delta_1$.  Furthermore, summing with the polynomial corresponding to $\delta_1$ cancels the terms not giving an intersection with $\alpha_3$ as in Figure~\ref{exception}. Thus $|S_i(-1)|$ gives the geometric intersection number of $\beta(y_i)$ and $\alpha_3$.  
\begin{figure}[htbp]
\begin{center}
\includegraphics[height=1in]{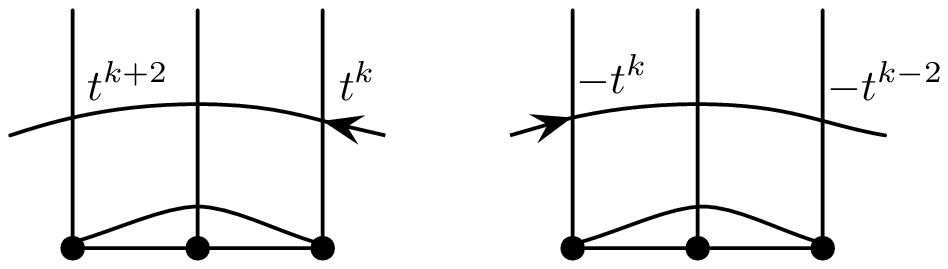}
\caption{}
\label{exception}
\end{center}
\end{figure}

Now, given these intersection numbers we argue that we can determine the action of $\beta$ on each $y_i$ up to full Dehn twists about the boundary.

\subsection{Proof of Theorem~\ref{same}}

Let $\beta\in B_3$.  Since $\beta$ is uniquely determined by its action on the loops $y_1,y_2,y_3$, we argue that the image of each $y_i$ under $\beta$ can be determined up to the action of the braid $\Delta^{2k}$ and isotopy given  the geometric intersection numbers $i(\beta(y_i),\alpha_1)=a_i$, $i(\beta(y_i),\alpha_2)=b_i$, and $i(\beta(y_i),\alpha_3)=c_i$.  Without any loss of generality, we consider the image $\beta(y_1)$.  The arc $\beta(y_1)$ is the union of subarcs inside and outside the triangle formed by $\alpha_1,\alpha_2,\alpha_3$.  Two arcs of particular interest are the initial subarc and the terminal subarc which intersect the boundary of the disk (shown in red).  

Conceivably, the inside and outside of the triangle looks like the two train track diagrams pictured below in Figure~\ref{inandout}.  
\begin{figure}[htbp]
\begin{center}
\subfigure[Inside the triangle]{\label{triangle1}\includegraphics[width=1in]{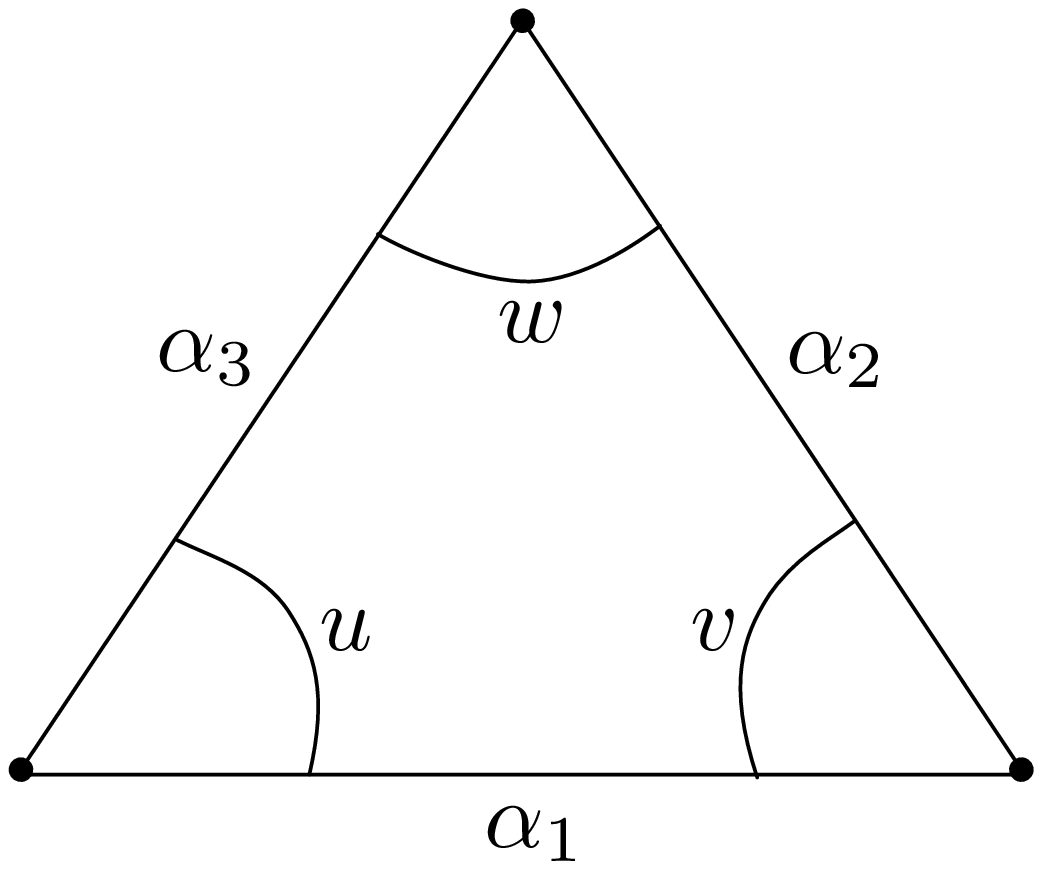}} 
\subfigure[Outside the triangle]{\label{triangle2}\includegraphics[width=1in]{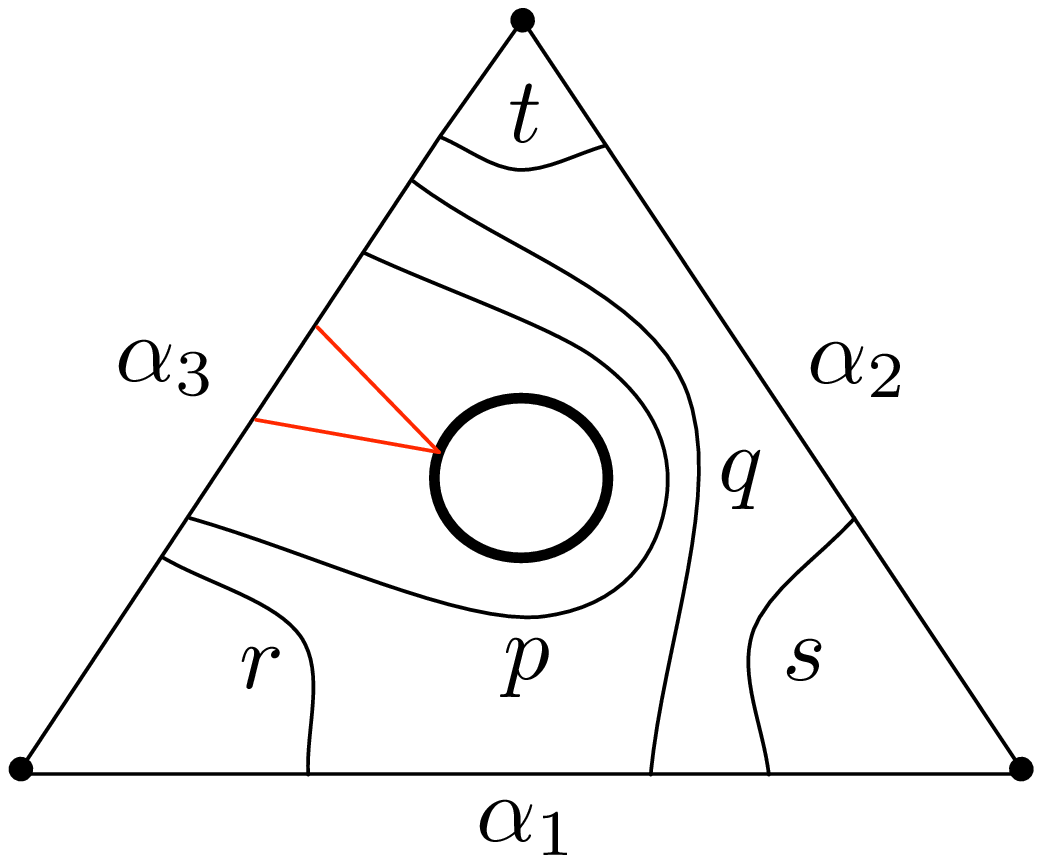}} 
\end{center} \caption{}
\label{inandout}
\end{figure} 

However, we see that one of $u,v,w$ must be zero, for if all of these are positive, then all of $r,s,t$ must be zero (else, we would have a loop).  This would imply that one of the sides of the triangle has no intersection points contradicting $u,v,w$ all positive.  Hence, one of $u,v,w$ must be zero or equivalently, two of $a_1,b_1,c_1$ sum to give the third.  Without any loss of generality, let's say $a_1+b_1=c_1$.  With this assumption, we can conclude that the inside of the triangle is one of the two options shown in Figure~\ref{inside}.  
\begin{figure}[htbp]
\begin{center}
\subfigure[]{\label{inside1}\includegraphics[width=1in]{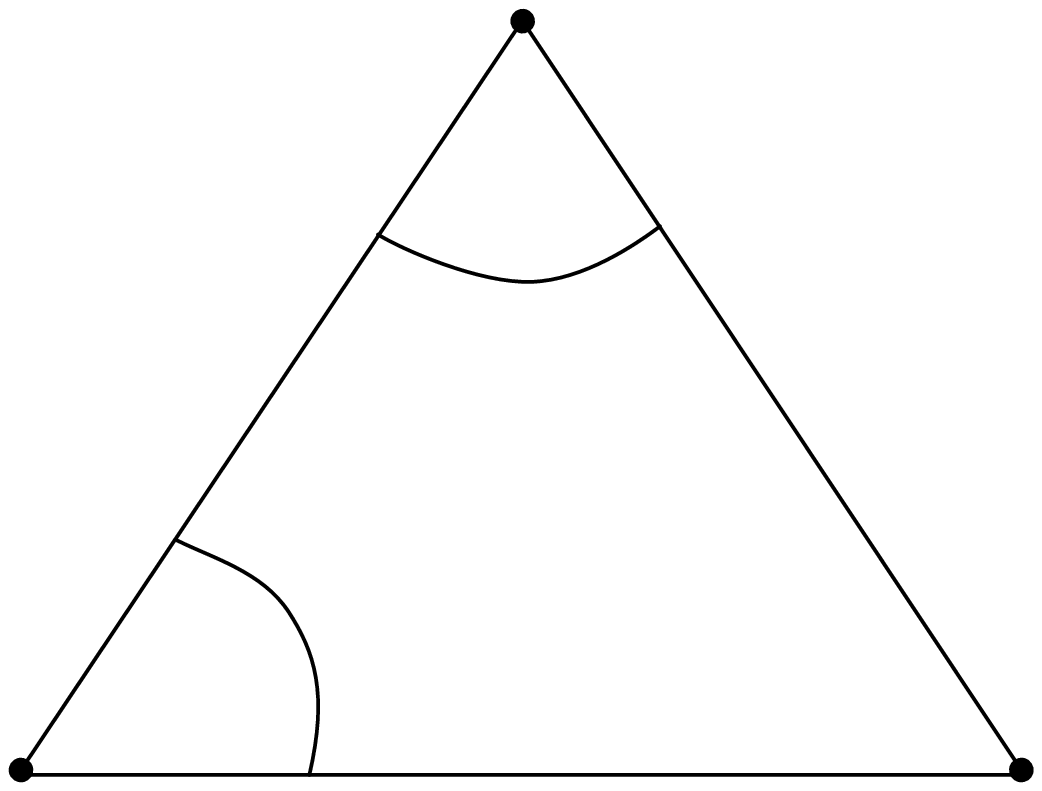}} 
\subfigure[]{\label{inside3}\includegraphics[width=1in]{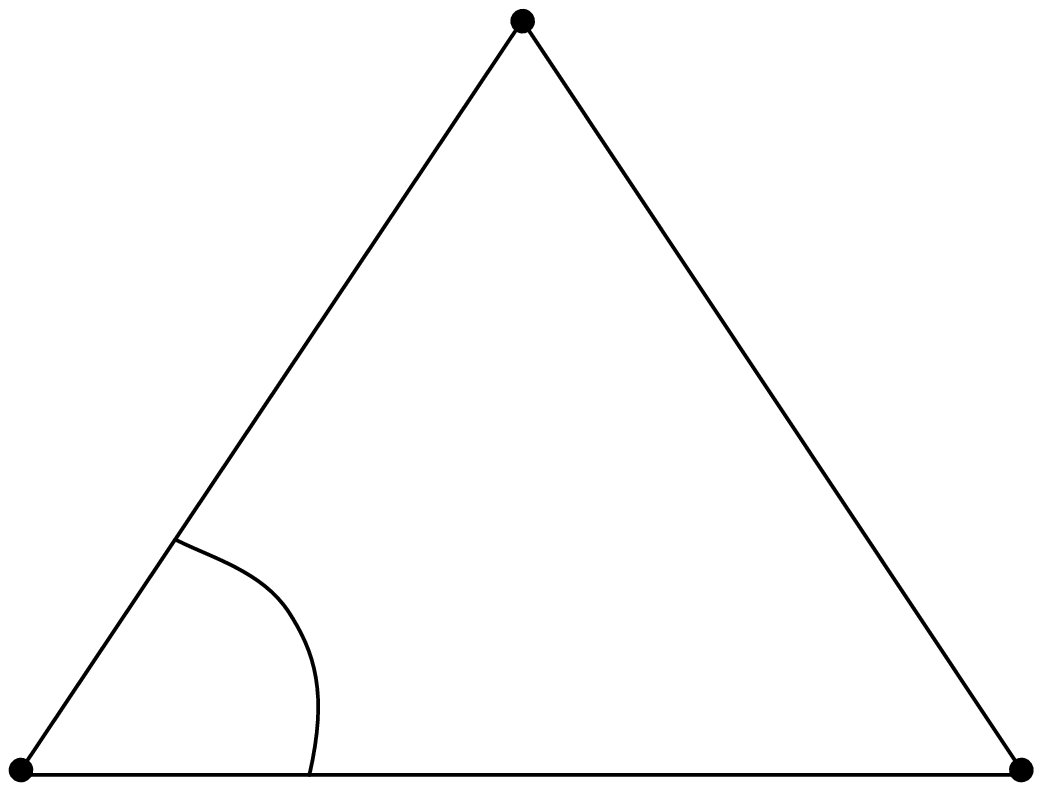}}  
\end{center} \caption{}
\label{inside}
\end{figure} 

It can be easily checked that the configuration in Figure~\ref{inside3} is the special case that $a_1=1, b_1=0, c_1=1$, and hence $\beta$ acts as the identity on $y_1$ up to multiplication by $\Delta^{2k}$.  Assuming that $a_1,b_1,c_1$ are all positive as in Figure~\ref{inside1}, we have the following six choices for the outside of the triangle shown in Figure~\ref{outside}.

\begin{figure}[htbp]
\begin{center}
\subfigure[]{\label{outside1a}\includegraphics[width=1in]{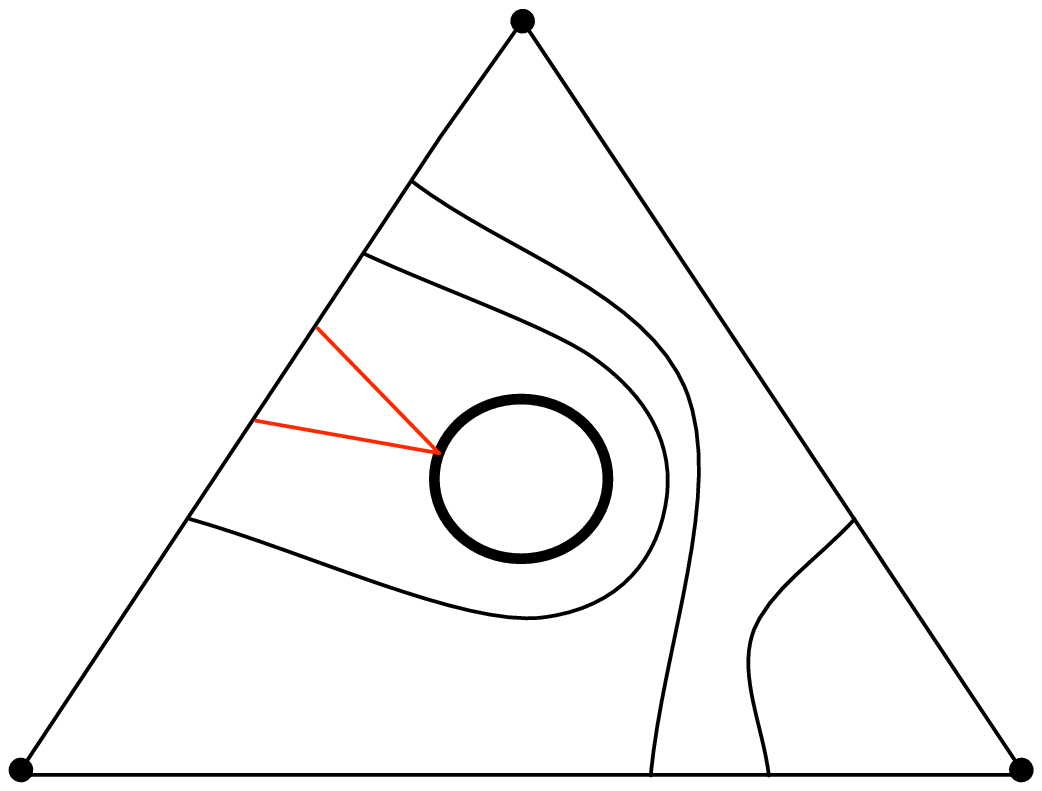}} 
\subfigure[]{\label{outside1b}\includegraphics[width=1in]{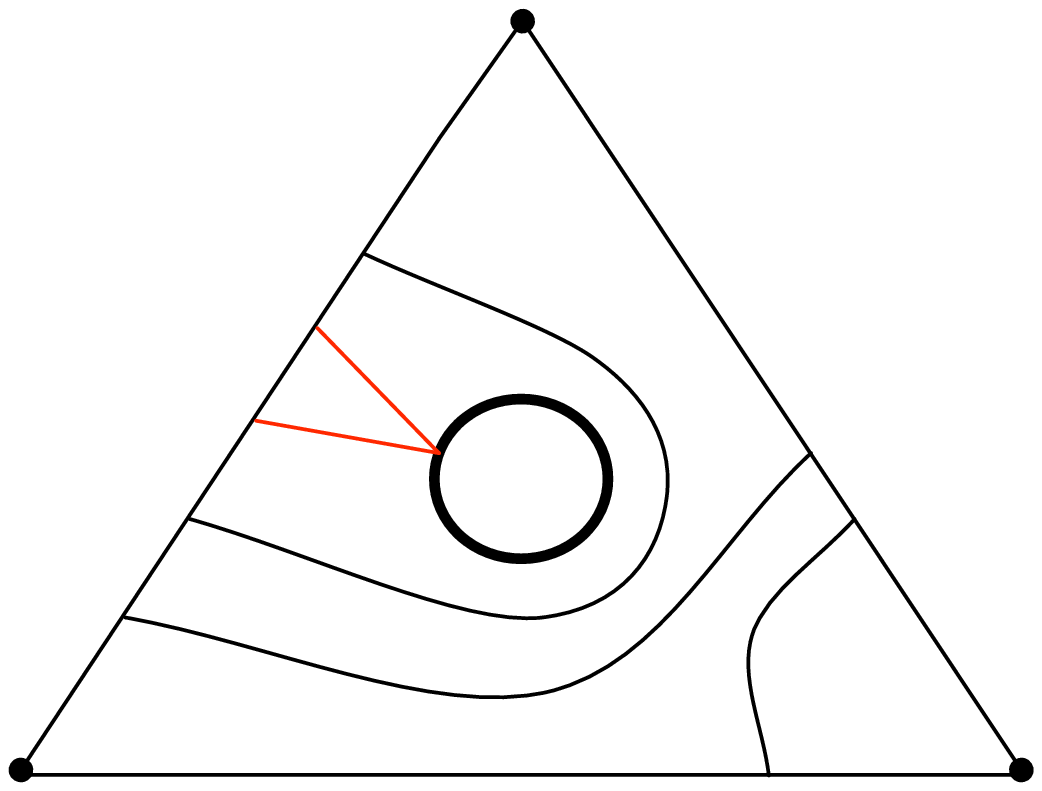}}
\subfigure[]{\label{outside2a}\includegraphics[width=1in]{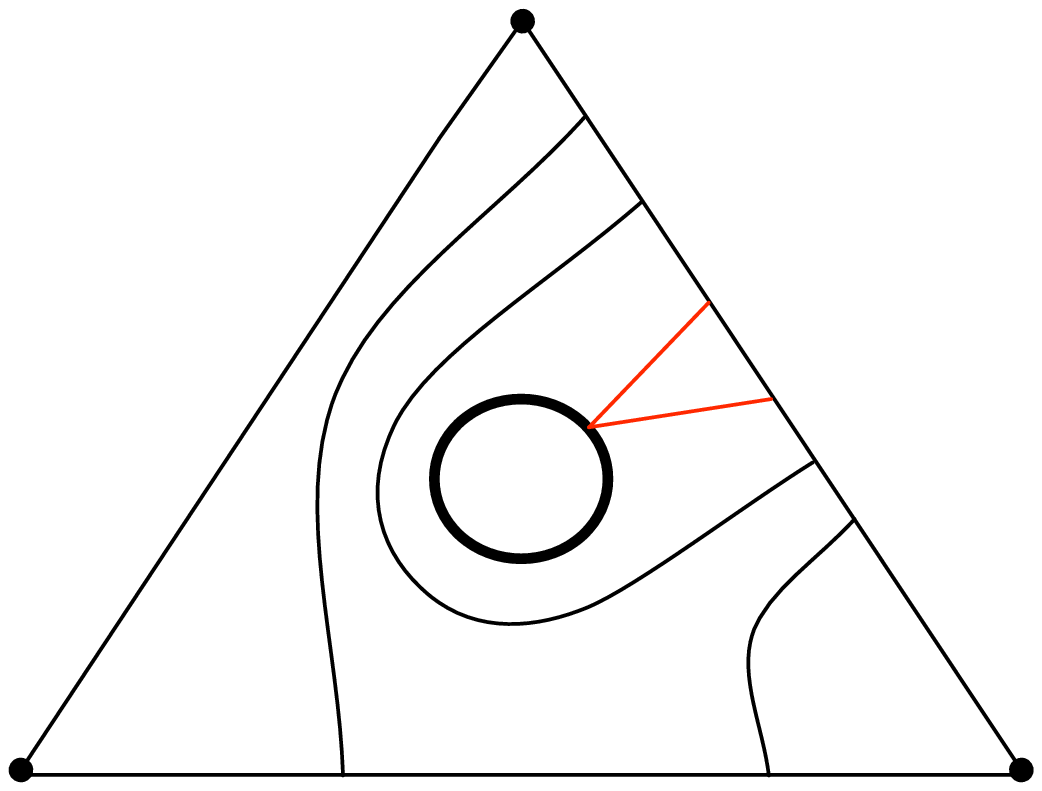}}
\subfigure[]{\label{outside2b}\includegraphics[width=1in]{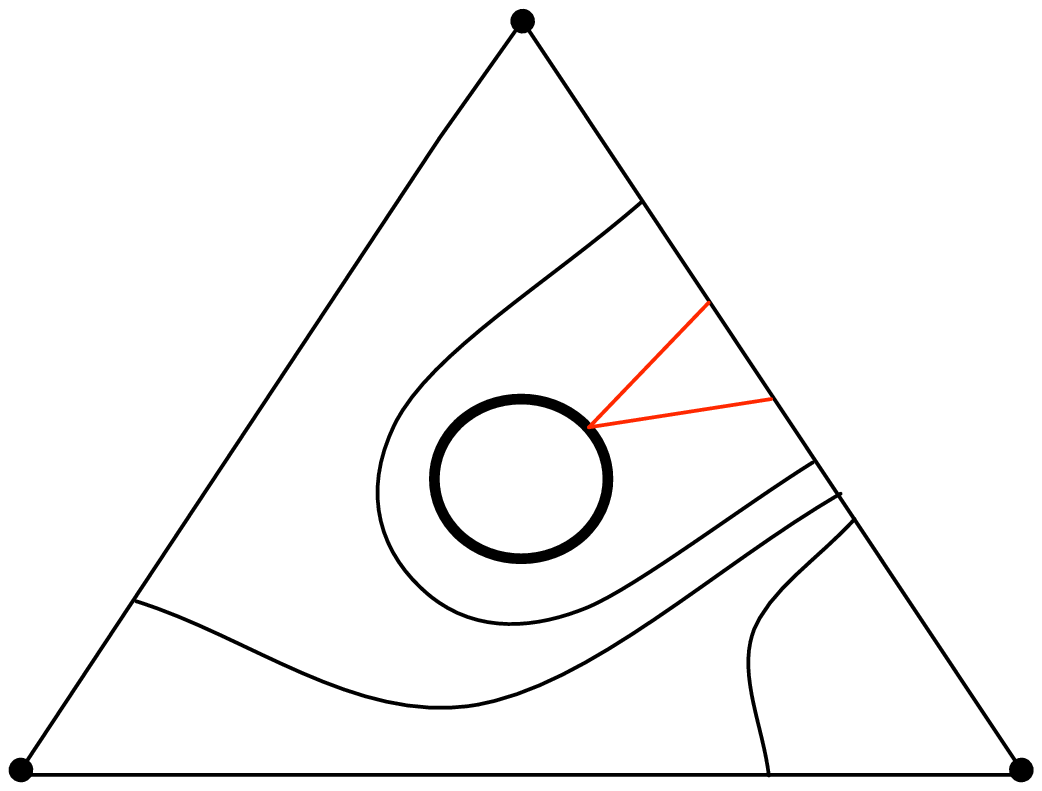}} 
\subfigure[]{\label{outside3a}\includegraphics[width=1in]{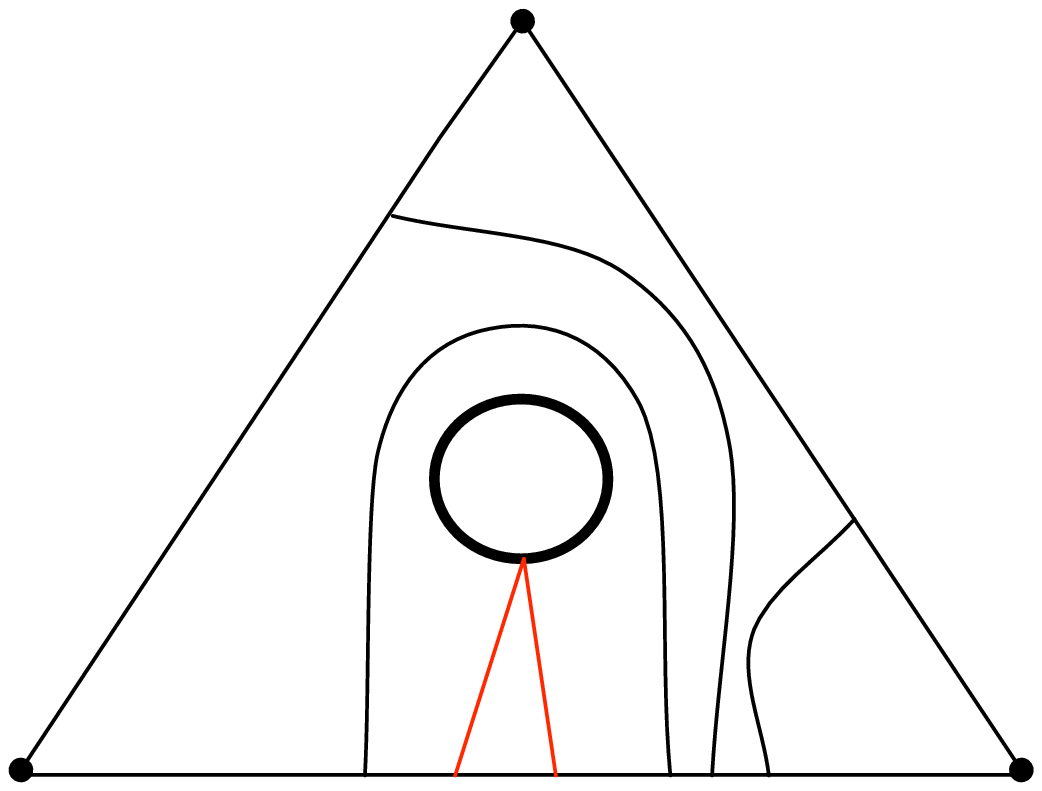}}
\subfigure[]{\label{outside3b}\includegraphics[width=1in]{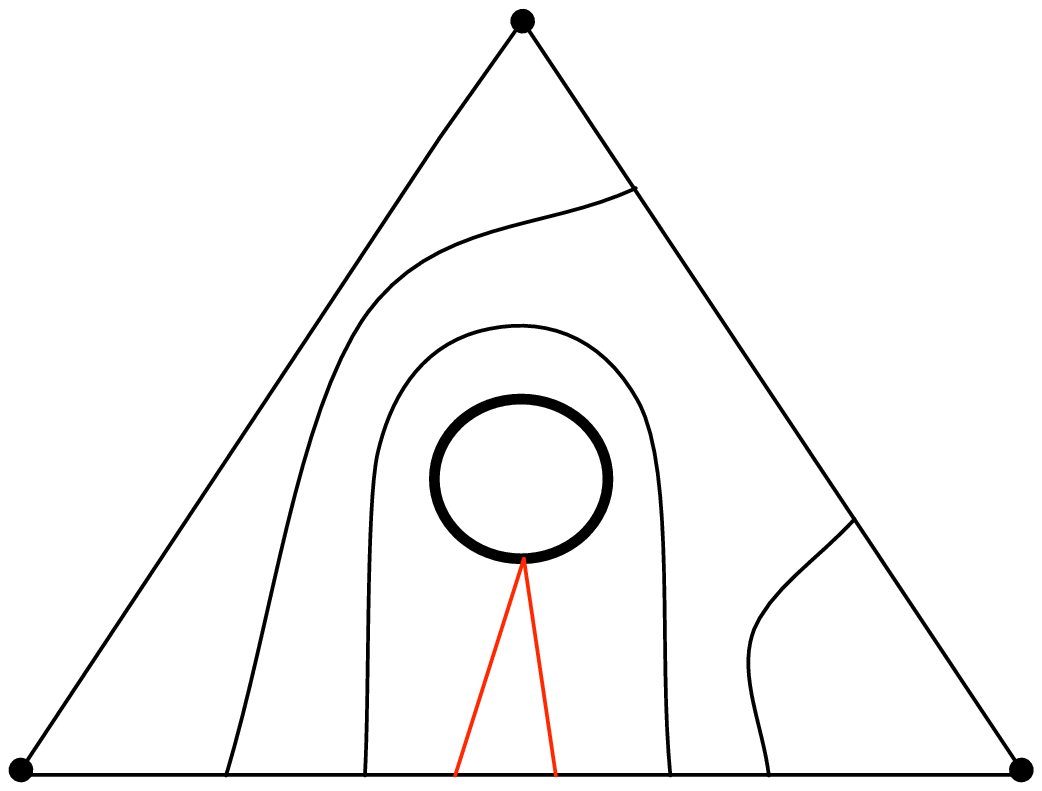}}    
\end{center} \caption{}
\label{outside}
\end{figure} 

Immediately, we can rule out the configurations in Figures~\ref{outside2a} and ~\ref{outside3b} because each would imply that $c_1=0$.  Now consider the configuration in Figure~\ref{outside2b}.  If $n_1,n_2,n_3$ are the number of points on each side shown below in Figure~\ref{outside2blabeled}, then $n_1=a_1$, $n_2=c_1$, and $n_1+n_2+n_3=b_1$.  Given that $a_1+b_1=c_1$, this gives that $n_1$ and $n_3$ are both equal to zero.  This is a contradiction since $a_1$ is not equal to zero.  We can also eliminate configuration ~\ref{outside3a} in a similar fashion (they are mirror images of each other).

\begin{figure}[htbp]
\begin{center}
\includegraphics[width=1.5in]{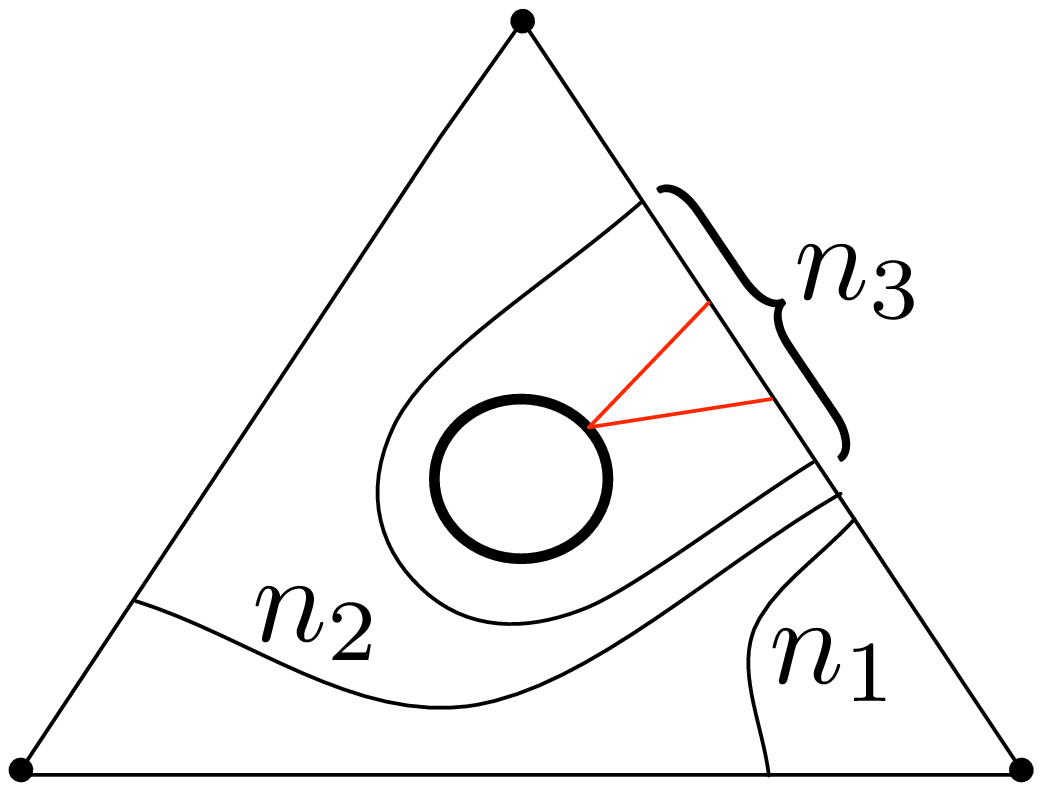}
\caption{}
\label{outside2blabeled}
\end{center}
\end{figure}

Therefore, we are left with only two possibilities: ~\ref{outside1a} and ~\ref{outside1b}.  If $m_1,m_2,m_3$ are as shown below in Figure~\ref{outside1alabeled}, then $m_1=b_1, m_1+m_2=a_1$, and $m_2+m_3=c_1$.  In particular, this configuration corresponds to the case that $a_1>b_1$.  It is easily checked that ~\ref{outside1b} corresponds to the case that $b_1>a_1$, and if $a_1=b_1$, then $m_2=0$.  Hence, the only ambiguity is the amount of twisting that occurs in a neighborhood of $\partial D_n$.  Thus, we can uniquely determine each $\beta(y_i)$ up to multiplication by powers of $\Delta^2$.

\begin{figure}[htbp]
\begin{center}
\includegraphics[width=1.5in]{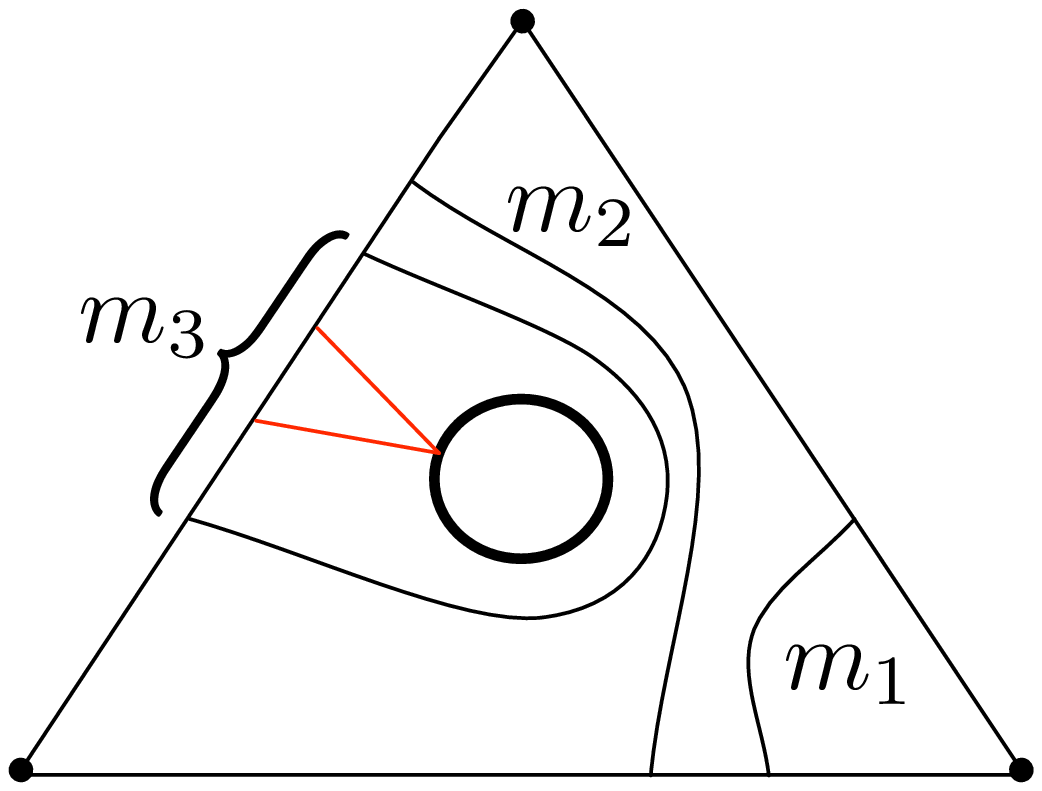}
\caption{}
\label{outside1alabeled}
\end{center}
\end{figure}

\subsection{Veering information depends on k}
Given $\beta \in B_3$, for each $i=1,2,3$ we can construct $y_i'$ in $\pi_1(D_n,p_0)$ so that the following hold:
\begin{enumerate}
\item[i.] $y_i'$ and $\beta(y_i)$ have the same geometric intersection numbers $a_i,b_i,c_i$;
\item[ii.] $y_i'$ is to the right of $y_i$; and
\item[iii.] $y_i'$ has minimal boundary twisting in the sense that $\Delta^{-2}(y_i')$ is to the left of $y_i$.  
\end{enumerate}
Furthermore, Theorem~\ref{same} guarantees that $\beta(y_i)=\Delta^{2k}(y_i')$.   Therefore, the sign of $k$ tells us whether $\beta(y_i)$ is to the right of $y_i$.  We summarize this in the following corollary.

\begin{corollary} Let $y_i' \in \pi_1(D_n,p_0)$ be as described above.  Then $\beta(y_i)=\Delta^{2k}(y_i')$.  Moreover, $\beta(y_i)$ is to the right of $y_i$ if and only if $k\geq 0$.\end{corollary} 
\begin{proof} This is immediate from the construction of $y_i'$ and Theorem~\ref{same}.\end{proof}

Now, how do we find the integer $k$?  Admittedly, this is the most cumbersome part of the process.  After $y_i'$ is found, we use the method described in Section~\ref{geometric} to calculate each corresponding column vector $C_i$ of a matrix which we call $\mathrm{M(\beta')}$.  Here is where we use the faithfulness of the representation for $n=3$.  If we are lucky, this matrix will be $\mathrm{M(\beta)}$, and hence $k=0$.  If we are not, we use our computing software of choice to multiply $\mathrm{M(\beta')}$ by powers of $\mathrm{M(\Delta^2)}= \begin{pmatrix}t^2&t^2-t^3&t^2-t^3\\ t-t^2&t-t^2+t^3&t-t^2\\ 1-t&1-t&1-t+t^3\end{pmatrix}$ until we reach $\mathrm{M(\beta)}$.  The power of $\mathrm{M(\Delta^2)}$ that gives $\mathrm{M(\beta)}$ is equal to $k$.

{}

\end{document}